\documentclass[a4paper,11pt,english]{article}
\usepackage{babel}
\usepackage{authblk} 
\usepackage[latin1]{inputenc}
\usepackage[pagewise]{lineno}
\usepackage[percent]{overpic} 

\author[1]{Juli\'an L\'opez-G\'omez\thanks{Corresponding author: jlopezgo@ucm.es}}
\author[1]{Alejandro Sahuquillo\thanks{alejsahu@ucm.es}}
\affil[1]{\small Universidad Complutense de Madrid\\ Departamento de An\'alisis Matem\'atico y Matem\'atica Aplicada\\ Plaza de las Ciencias 3\\ 28040 Madrid, Spain}

\title{\textbf{Limiting behavior of principal eigenvalues for a class of mixed boundary value problems as the measure of the support domain goes to zero }}

\date{\today}

\usepackage{amsmath,amsfonts,amssymb}
\usepackage{yhmath,mathrsfs,yfonts,arcs}
\usepackage{amsthm}
\usepackage{multirow}
\usepackage{graphicx}
\usepackage[percent]{overpic} 
\usepackage{paralist}
\usepackage[titletoc,title]{appendix}
\usepackage{cite}
\usepackage{leftidx}
\usepackage{soul}
\usepackage{color}
\usepackage{bbm}
\usepackage{mathtools}
\mathtoolsset{showonlyrefs}
\usepackage{subcaption}
\usepackage{enumitem}
\usepackage{hyperref}

\theoremstyle{plain}
\newtheorem{theorem}{Theorem}[section]
\newtheorem{proposition}[theorem]{Proposition}

\theoremstyle{definition}

\theoremstyle{remark}
\numberwithin{equation}{section}

\usepackage[margin=0.80in]{geometry}

\newcommand\Item[1][]{%
	\ifx\relax#1\relax  \item \else \item[#1] \fi
	\b_0bovedisplayskip=0pt\b_0bovedisplayshortskip=0pt~\vspace*{-\baselineskip}}

\newcommand{\field}[1] {\mathbb{#1}}

\newcommand{\R}{\field{R}}

\def\a{\alpha}
\def\b{\beta}

\def\D{\Delta}

\def\G{\Gamma}

\def\o{\omega}
\def\O{\Omega}
\def\p{\partial}

\def\s{\sigma}
\def\t{\theta}

\def\v{\varphi}

\def\ua{\uparrow}
\def\da{\downarrow}

\def\vt{\vartheta}

\newcommand{\mc}{\mathcal}

\makeatletter
\newcommand{\pushright}[1]{\ifmeasuring@#1\else\omit\hfill$\displaystyle#1$\fi\ignorespaces}
\newcommand{\pushleft}[1]{\ifmeasuring@#1\else\omit$\displaystyle#1$\hfill\fi\ignorespaces}
\makeatother

\begin{document}
\maketitle

\begin{abstract}

\noindent In this paper we characterize the limiting behavior of the principal eigenvalue, $\s_1[-\D,\b,\O]$, of the boundary value problem \eqref{1.1} as the Lebesgue measure of the underlying domain, $\O$, tends to zero. Naturally, the domains $\O$ are assumed to be included on a fixed open set $D$ such that $\b\in\mc{C}(D)$, and they satisfy $\bar\O\subset D$. Our main result establishes that, in the classical case when $\inf_{D}\b >0$,
$$
  \lim_{|\O|\da 0}\s_1[-\D,\b,\O] =+\infty,
$$
whereas
$$
  \lim_{|\O|\da 0}\s_1[-\D,\b,\O] =-\infty\;\;\hbox{if}\;\; \sup_{D}\b <0,
$$
which is a surprising result at the light of the classical existing theorems for Dirichlet boundary conditions.
Furthermore, in the special case when $\b>0$ is a constant, we can prove that 
$$
  \lim_{R\da 0}\left( R \s_1[-\D,\b,B_R]\right)=\b \frac{\mathrm{Area}(\p B_1)}{|B_1|},
$$
where,  we are denoting $B_\varrho:=\{x\in\R^N\;:\;|x|<\varrho\}$ for all $\varrho>0$. 
\end{abstract}



\smallskip
\noindent \textbf{Acknowledgements:} This work has been supported by the Ministry of Science and Innovation of Spain under Research Grant PID2024-155890NB-I00.

\section{Introduction}
\label{sec:1}
\noindent This paper analyzes the limiting behavior of the principal eigenvalue, $\s_1[-\D,\b,\O]$, of the linear eigenvalue problem
\begin{equation}
\label{1.1}
	\left \{ \begin{array}{ll}
		-\D u=\s u \quad & \hbox{in} \;\; \O,\\[1ex]
		\frac{\p u}{\p n} + \b(x) u = 0 \quad   & \hbox{on} \;\; \p \O,  \end{array} \right.
\end{equation}
as $|\O|\da 0$, where $\O$ is a bounded domain of class $\mc{C}^{2,\a}$  of $\R^N$ ($N\geq 2$) for some $\a \in (0,1)$, $|\O|$ denotes the Lebesgue measure of $\O$, $n$ stands for the outward unit normal vector field along $\p\O$. Throughout this paper, $D \subset \R^N$ is an open set such that $\bar \O \subset D$ and $\b\in \mc{C}(D)$ is an arbitrary continuous function.
\par
According to, e.g., Theorem 5.1 of \cite{LG96}, or Proposition 8.6 of \cite{LG13},  which are based on the celebrated inequality of Faber \cite{F23} and Krahn \cite{K25}, it is already known that the principal eigenvalue of the associated \emph{Dirichlet} problem
\begin{equation*}
	\left \{ \begin{array}{ll}
		-\D u=\s u \quad & \hbox{in} \;\; \O,\\[1ex]
		u = 0 \quad   & \hbox{on} \;\; \p \O,  \end{array} \right.
\end{equation*}
denoted by $\s_1[-\D,\mc{D},\O]$, satisfies
\begin{equation}
\label{1.2}
  \lim_{|\O|\da 0}\s_1[-\D,\mc{D},\O]=+\infty.
\end{equation}
Naturally, this property fails to be true for the \emph{Neumann} problem
\begin{equation*}
	\left \{ \begin{array}{ll}
		-\D u=\s u \quad & \hbox{in} \;\; \O,\\[1ex]
	\frac{\p u}{\p n} = 0 \quad   & \hbox{on} \;\; \p \O,  \end{array} \right.
\end{equation*}
as, in such a case, the underlying principal eigenvalue, denoted by $\s_1[-\D,\mc{N},\O]$ in this paper, equals zero regardless the size of $\O$. Up to the best of our knowledge, the behavior of the principal eigenvalue of \eqref{1.1} as $|\O|\da 0$ remains an open problem even in the case when $\b$ is assumed to be constant on $\p\O$.
Our first result provides us with that limiting behavior in the simplest one-dimensional prototype model
\begin{equation}
\label{1.3}
	\left \{ \begin{array}{ll}
		-u''=\s u \qquad  \hbox{in} \;\; [0,L],\\[1ex]
		\mc{B}_0 u(0) = 0, \quad \mc{B}_L u(L) = 0, \end{array} \right.
\end{equation}
as $L \da 0$, where $\mc{B}_0,\; \mc{B}_L  \in \{ \mc{D}, \mc{N}, \mc{R}_\b \}$ are boundary operators of some of the following types: either $\mathcal{D}u = u$, or  $\mathcal{N}u = u'$, or
$$
   \mathcal{R}_\b u = \frac{\p u}{\p n} + \b u :=
   \left\{ \begin{array}{ll} -u'(0)+\b_0 u(0) & \quad \hbox{at}\;\;x=0,\\[1ex]
    u'(L)+\b_\o u(L) & \quad \hbox{at}\;\;x=L,\end{array}\right.
$$
where $\b_0$ and $\b_\o$ are two arbitrary constants, not necessarily positive as in the classical setting of the Sturm--Liouville theory. Note that $\mc{R}_\b=\mc{N}$ at $x=0$ (resp.
$x=L$) if $\b_0=0$ (resp. $\b_\o=0$). Actually, we can always assume that $\mc{B}_0=\mc{R}_\b$ at $x=0$ and
 $\mc{B}_L=\mc{R}_\b$ at $x=L$ by adopting the convention that
$$
    \b_0=+\infty\;\; \hbox{if}\;\; \mc{B}_0=\mc{D}, \;\; \hbox{and}\;\; \b_\o=+\infty
    \;\; \hbox{if}\;\; \mc{B}_L=\mc{D}.
$$
With these conventions in mind, throughout this paper, the principal eigenvalue of \eqref{1.3} will be denoted by
$\s_1[-D^2,\b_0,\b_\o,(0,L)]$. Note that problem \eqref{1.3} cannot be thought as a particular case of \eqref{1.1}, for some choice of $\b(x)$, because we are not imposing that $\b_0=\b_\o$. Indeed, in the one-dimensional setting of problem \eqref{1.1}, if $\O=(0,L)$, the continuity of $\b$ entails that
$$
  \lim_{L\da 0} \b(L)=\b(0),
$$
whereas in the setting of \eqref{1.3}, in general,  $\b_0\neq \b_\o$.
\par
The main result in the one-dimensional case can be stated as follows. It characterizes
whether, or not, the condition \eqref{1.2} holds true for the model \eqref{1.3} under general boundary conditions
of mixed type, not necessarily of Sturm--Liouville type.

\begin{theorem}
\label{th1.1}
The following conditions are satisfied.
\begin{enumerate}
\item[{\rm (a)}] If $\b_0 + \b_\o >0$, then
$$
\lim_{L \da 0} \s_1[-D^2,\b_0,\b_\o,(0,L)] = + \infty.
$$
\item[{\rm (b)}] If $\b_0+ \b_\o =0$, then
$$
\s_1[-D^2,\b_0,\b_\o,(0,L)] = - \b_0^2 \leq 0 \;\; \hbox{for all} \; \; L>0.
$$
\item[{\rm (c)}] If $\b_0 + \b_\o< 0$, then
$$
\lim_{L \da 0} \s_1[-D^2,\b_0,\b_\o,(0,L)] = - \infty.
$$
\end{enumerate}
\end{theorem}
\par
The behavior of Part (c) has not been previously described in the literature. Possibly, because
in the setting of the Sturm--Liouville theory $\b_0\geq 0$ and $\b_\o\geq 0$.
\par
The main result of this paper in the multidimensional case can be stated as follows.
\begin{theorem}
\label{th1.2}
Let $\s_1[-\D,\b,\O]$ be the principal eigenvalue of \eqref{1.1}. Then,
\begin{equation}
\label{1.4}
\lim_{|\O| \da 0} \s_1[-\D,\b,\O] = + \infty \quad  \hbox{if} \; \; \inf_{D}\b >0.
\end{equation}
Moreover,
\begin{equation}
\label{1.5}
\lim_{|\O| \da 0} \s_1[-\D,\b,\O] = - \infty \quad \hbox{if} \; \; \sup_{D}\b <0.
\end{equation}
\end{theorem}
\par
The distribution of this paper is the following. Section 2 delivers some preliminaries,  among them, the monotonicity properties of the principal eigenvalues used throughout this paper. Section 3 is devoted to the proof of Theorem \ref{th1.1}, which relies on some elementary integration techniques. In Section 4 we deliver the proof of Theorem \ref{th1.2}, which is based on some sophisticated comparison techniques combined with some technical devices introduced in Lemma 2.1 and Theorem 2.1 of \cite{LG13}. Finally, in Section 5 we find out the asymptotic expansion of
$\s_1[-\D,\b,B_R]$ as $R\da 0$ in the special case when $\b>0$ is a constant. It turns out that
$$
  \lim_{R\da 0}\left( R \s_1[-\D,\b,B_R]\right)=\b \frac{\mathrm{Area}(\p B_1)}{|B_1|}.
$$

\section{Preliminaries}
\label{sec2}
This section collects some preliminary results, mainly based on comparison principles, that will be used
throughout this article. We begin by recalling some general properties of principal eigenvalues and their associated eigenfunctions.
\par
Subsequently, we will denote by $\v_1$ any  principal eigenfunction associated to $\s_1[-\D,\b,\O]$, and by $\psi_1$ any principal eigenfunction of \eqref{1.3} associated with $\s_1[-D^2,\b_0,\b_\o,(0,L)]$. According to Sections 7 and 8 of \cite{LG13}, we already know that:
\begin{itemize}
\item $\s_1[-\D,\b,\O]$ and $\s_1[-D^2,\b_0,\b_\o,(0,L)]$ are the only eigenvalues of \eqref{1.1} and \eqref{1.3} that admit some positive eigenfunction;

\item $\v_1$ and $\psi_1$ are unique, up to positive multiplicative constants, and they satisfy $\v_1 \in \mc{C}^1(\bar{\O}) \cap \mc{C}^2(\O)$, $\v_1(x) > 0$ for all $x \in \bar{\O}$,
$\psi_1 \in \mc{C}^\infty[0,L]$, $\psi_1(x) >0$ for all $x \in (0,L)$, $\psi_1(0) >0$ if $\mc{B}_0\in \{\mc{N},\mc{R}_\b:\b\in\R\}$, $\psi_1'(0) > 0$ if $\mc{B}_0 = \mathcal{D}$, and
$\psi_1(L) >0$ if $\mc{B}_L\in \{\mc{N},\mc{R}_\b:\b\in\R\}$ and $\psi_1'(L) < 0$ if $\mc{B}_L = \mathcal{D}$. These properties can be summarized by saying that $\v_1$ and $\psi_1$ are strongly positive ($\v_1,\psi_1 \gg 0$);

\item any other eigenvalues of \eqref{1.1} and \eqref{1.3} are real and greater than their respective principal eigenvalues $\s_1[-\D,\b,\O]$ and $\s_1[-D^2,\b_0,\b_\o,(0,L)]$. Thus, they are strictly dominant.
\end{itemize}
Most of these properties  go back to Cano-Casanova and L\'opez-G\'omez \cite{CCLG} for general boundary conditions.
\par
Throughout this paper, a \emph{supersolution} of the problem \eqref{1.1} is any function $\overline{u}\in\mc{C}^1(\bar{\O}) \cap \mc{C}^2(\O)$ such that
\begin{equation}
	\left \{ \begin{array}{ll}
		-\D\overline{u} \geq \s \overline{u} \quad & \hbox{in} \;\; \O,\\[1ex]
		\frac{\p \overline{u}}{\p n} + \b(x) \overline{u} \geq 0 \quad   &\hbox{on} \;\; \p \O.  \end{array} \right.
\end{equation}
In such a case, $\overline{u}$ is said to be a \emph{strict supersolution} if some of these inequalities is strict. Similarly, a \emph{supersolution} of \eqref{1.3} is any function $\overline{u}\in \mc{C}^2[0,L]$ such that
\begin{equation}
	\left \{ \begin{array}{ll}
		-\overline{u}'' \geq \s \overline{u} \qquad  \hbox{in} \;\; [0,L],\\[1ex]
		\mc{B}_0 \overline{u}(0) \geq 0, \quad \mc{B}_L \overline{u}(L) \geq 0, \end{array} \right.
\end{equation}
and $\overline{u}$ is said to be a \emph{strict supersolution} if some of these inequalities is strict. The following results establish the main monotonicity properties of $\s_1[-D^2, \b_0,\b_\o,(0,L)]$ and $\s_1[-\D,\b,\O]$.

\begin{proposition}
\label{pr2.1}
The principal eigenvalue of \eqref{1.3} is increasing with respect to $\b_0\in (-\infty,+\infty]$ and
$\b_\o\in (-\infty,+\infty]$, in the sense that
\begin{equation}
\label{2.1}
\s_1[-D^2, \b_0^2 ,\b_\o^2,(0,L)] > \s_1[-D^2, \b_0^1 ,\b_\o^1,(0,L)]  \quad \hbox{if} \; \; \b_0^2 \geq  \b_0^1,\;\;
\b_\o^2 \geq  \b_\o^1, \;\; (\b_0^2,\b_\o^2)\neq (\b_0^1,\b_\o^1).
\end{equation}
\end{proposition}
\begin{proof}
Assume that $u_1\gg 0$ and $u_2\gg 0$ are principal eigenfunctions associated to the eigenvalue problems
\begin{equation}
	\left \{ \begin{array}{ll}
		-u_1''=\s_1[-D^2, \b_0^1 ,\b_\o^1,(0,L)] u_1  \qquad  \hbox{in} \;\; [0,L],\\[1ex]
		-u_1'(0)+\b_0^1 u_1(0)= 0, \quad u_1'(L)+\b_\o^1 u_1(L)  = 0, \end{array} \right.
\end{equation}
and
\begin{equation}
	\left \{ \begin{array}{ll}
		-u_2''=\s_1[-D^2, \b_0^2 ,\b_\o^2,(0,L)] u_2  \qquad  \hbox{in} \;\; [0,L],\\[1ex]
		-u_2'(0)+\b_0^2 u_2(0)= 0, \quad  u_2'(L)+\b_\o^2 u_2(L) = 0,  \end{array} \right.
\end{equation}
respectively. In the proof we will differentiate among several cases.
\par
\vspace{0.2cm}
\noindent \emph{Case 1: $\b_0^1, \b_\o^1\in\R$.} Then, by a direct calculation, we find that
$$
-u_1'(0)+\b_0^2 u_1(0) =\left(\b_0^2 - \b_0^1\right)u_1(0) \geq 0 \quad \hbox{and} \quad u_1'(L)+\b_\o^2 u_1(L) = \left(\b_\o^2 - \b_\o^1\right)u_1(L) \geq 0,
$$
even in the case when $\b_0^2$, or $\b_\o^2$, equals $+\infty$. Since $\b_0^1, \b_\o^1\in\R$, we have that $u_1(0)>0$ and $u_1(L)>0$. Hence, one of the previous inequalities is strict because $(\b_0^2,\b_\o^2)\neq (\b_0^1,\b_\o^1)$. Thus,
$u_1$ is a positive strict supersolution of the problem
\begin{equation}
	\left \{ \begin{array}{ll}
		-u''=\s_1[-D^2, \b_0^1 ,\b_\o^1,(0,L)] u  \qquad  \hbox{in} \;\; [0,L],\\[1ex]
		-u'(0)+\b_0^2 u(0)= 0, \quad  u'(L)+\b_\o^2 u(L)  = 0. \end{array} \right.
\end{equation}
By the notational conventions in Section 1, $-u'(0)+\b_0^2 u(0)  = 0$ becomes into $u(0)=0$ if $\b_0^2=+\infty$. Similarly, $u'(L)+\b_\o^2 u(L)  = 0$ becomes into $u(L)=0$ if $\b_\o^2=+\infty$. Consequently, it follows from \cite[Th. 7.10]{LG13} (going back to L\'opez-G\'omez and
Molina-Meyer \cite{LGMM} and Amann and L\'opez-G\'omez \cite{ALG}) that
$$
\s_1[-D^2-\s_1[-D^2, \b_0^1 ,\b_\o^1,(0,L)], \b_0^2 ,\b_\o^2,(0,L)]>0.
$$
Therefore, \eqref{2.1} holds. This ends the proof in this case.
\par
\vspace{0.2cm}
\noindent \emph{Case 2: $\b_0^1=+\infty$ and $\b_\o^1\in\R$.} This entails $\b_0^2=+\infty$. Thus, by the convention adopted in Section 1, $u_1(0)=u_2(0)=0$. Moreover, since $\b_\o^1\in\R$, we find that $u_1(L)>0$ and
$$
  u_1'(L)+\b_\o^2 u_1(L) = \left(\b_\o^2 - \b_\o^1\right)u_1(L) \geq 0.
$$
On the other hand, since $\b_0^1=\b_0^2=+\infty$ and $(\b_0^2,\b_\o^2)\neq (\b_0^1,\b_\o^1)$, necessarily
$\b_\o^2>\b_\o^1$. Hence,
$$
  u_1'(L)+\b_\o^2 u_1(L) = \left(\b_\o^2 - \b_\o^1\right)u_1(L) > 0.
$$
Consequently, $u_1$ is a positive strict supersolution of the problem
\begin{equation}
	\left \{ \begin{array}{ll}
		-u''=\s_1[-D^2, +\infty,\b_\o^1,(0,L)] u  \qquad  \hbox{in} \;\; [0,L],\\[1ex]
		u(0)= 0, \quad  u'(L)+\b_\o^2 u(L)  = 0. \end{array} \right.
\end{equation}
Note that $u'(L)+\b_\o^2 u(L)  = 0$ becomes into $u(L)=0$ if $\b_\o^2=+\infty$. Therefore, in any circumstances,
it follows from \cite[Th. 7.10]{LG13} that
$$
\s_1[-D^2-\s_1[-D^2,+\infty,\b_\o^1,(0,L)], +\infty ,\b_\o^2,(0,L)]>0,
$$
which implies \eqref{2.1}.
\par
\vspace{0.2cm}
\noindent \emph{Case 3: $\b_0^1\in\R$ and $\b_\o^1=+\infty$.} This entails $\b_\o^2=+\infty$ and $\b_0^2>\b_0^1$.
Thus, $u_1(0)>0$ and $u_1(L)=u_2(L)=0$. Hence,
$$
  -u_1'(0)+\b_0^2 u_1(0) = \left(\b_0^2 - \b_0^1\right)u_1(0) > 0.
$$
So, $u_1$ is a positive strict supersolution of the problem
\begin{equation}
	\left \{ \begin{array}{ll}
		-u''=\s_1[-D^2, \b_0^1,+\infty,(0,L)] u  \qquad  \hbox{in} \;\; [0,L],\\[1ex]
		-u'(0)+\b_0^2 u(0)=0, \quad  u(L) = 0. \end{array} \right.
\end{equation}
Observe that $-u'(0)+\b_0^2 u(0)  = 0$ becomes into $u(0)=0$ if $\b_0^2=+\infty$. Therefore,
it follows from \cite[Th. 7.10]{LG13} that
$$
\s_1[-D^2-\s_1[-D^2,\b_0^1,+\infty,(0,L)], \b_0^2, +\infty,(0,L)]>0,
$$
which implies \eqref{2.1}. The proof is complete.
\end{proof}

Similarly, the next result holds.
\begin{proposition}
\label{pr2.2}
The principal eigenvalue of the problem \eqref{1.1} is increasing with respect to $\b$, in the sense that
\begin{equation}
\s_1[-\D,\b_2 ,\O] > \s_1[-\D,\b_1 ,\O] \quad \hbox{if} \; \; \b_2 \gneq \b_1\;\;\hbox{on}\;\;\p\O.
\end{equation}
\end{proposition}
\begin{proof}
Let $\b_2 \gneq \b_1$, and $u_1 \gg  0$ and $u_2 \gg 0$ be principal eigenfunctions of
the eigenvalue problems
\begin{equation}
	\left \{ \begin{array}{ll}
		-\D u_1=\s_1[-\D,\b_1,\O]  u_1 \quad & \hbox{in} \;\; \O, \\[1ex]
		\frac{\p u_1}{\p n} + \b_1 u_1 = 0 \quad   & \hbox{on} \;\; \p \O, \end{array} \right.
\end{equation}
and
\begin{equation}
	\left \{ \begin{array}{ll}
		-\D u_2=\s_1[-\D,\b_2,\O]  u_2 \quad & \hbox{in} \;\; \O,\\[1ex]
		\frac{\p u_2}{\p n} + \b_2 u_2 = 0 \quad   & \hbox{on} \;\; \p \O,  \end{array} \right.
\end{equation}
respectively. Since $u_1(x)>0$ for all $x\in\bar \O$ and $\b_2 \gneq \b_1$, we have that
$$
\tfrac{\p u_1}{\p n} + \b_2 u_1 \gneq \tfrac{\p u_1}{\p n} + \b_1 u_1 =0.
$$
Thus, $u_1$ is a positive strict supersolution of the problem
\begin{equation}
	\left \{ \begin{array}{ll}
		-\D u=\s_1[-\D,\b_1,\O]  u \quad & \hbox{in} \;\; \O,\\[1ex]
		\frac{\p u}{\p n} + \b_2 u = 0 \quad   & \hbox{on} \;\; \p \O.  \end{array} \right.
\end{equation}
Therefore, again by \cite[Th. 7.10]{LG13}, we obtain that
$$
\s_1[-\D- \s_1[-\D,\b_1,\O],\b_2,\O] >0,
$$
which concludes the proof.
\end{proof}

The next result establishes a useful symmetry property of $\s_1[-D^2,\b_0,\b_\o,(0,L)]$.

\begin{proposition}
\label{pr2.3}
The principal eigenvalue of \eqref{1.3} is symmetric with respect to $\b_0\in(-\infty,+\infty]$ and $\b_\o\in(-\infty,+\infty]$, in the sense that
$$
\s_1[-D^2,\b_0,\b_\o,(0,L)] = \s_1[-D^2,\b_\o,\b_0,(0,L)].
$$
\end{proposition}
\begin{proof}
Let $u_1 \gg 0$ be a principal eigenfunction of \eqref{1.3}. Then, the change of variable
\begin{equation}
\label{2.2}
 v_1(y):= u_1(x) \quad \hbox{with} \; \; y := L-x, \; \; y \in [0,L],
\end{equation}
yields $u_1(0) = v_1(L)$ and $u_1(L) = v_1(0)$. Moreover, $u_1'(x) = -v_1'(y)$ and $u_1''(x) = v_1''(y)$ for all $x\in [0,L]$. In particular,
$$
  u_1'(0) = -v_1'(L),\quad u_1'(L) = -v_1'(0).
$$
Thus, the problem \eqref{1.3} becomes
\begin{equation}
\label{2.3}
	\left \{ \begin{array}{ll}
		-v_1''=\s_1[-D^2, \b_0 ,\b_\o,(0,L)] v_1  \qquad  \hbox{in} \;\; [0,L],\\[1ex]
		-v_1'(0)+\b_\o v_1(0)= 0, \quad  v_1'(L)+\b_0 v_1(L)  = 0. \end{array} \right.
\end{equation}
Since $u_1 \gg 0$, \eqref{2.2} implies that $v_1 \gg 0$. Hence, $v_1$ is a principal eigenfunction of \eqref{2.3} and, consequently, the uniqueness of the principal eigenvalue entails that
$$
\s_1[-D^2,\b_\o,\b_0,(0,L)] = \s_1[-D^2,\b_0,\b_\o,(0,L)].
$$
The proof is complete.
\end{proof}

\section{Proof of Theorem \ref{th1.1}}
\label{sec3}
The proof relies on the fact that the principal eigenvalue can be estimated explicitly. Note that, for every
$\s \in \R$, the general solution of the differential equation
\begin{equation}
\label{3.1}
-u'' = \s u
\end{equation}
is given by
\begin{equation}
	u(x) = \left \{ \begin{array}{ll}
		A \cos\left(\sqrt{\s}x\right) + B \sin\left(\sqrt{\s}x\right) \quad & \hbox{if} \;\; \s >0, \\[1ex]
		Ax+ B \quad   & \hbox{if} \;\; \s =0, \\[1ex]
		A \cosh\left(\sqrt{-\s}x\right) + B \sinh\left(\sqrt{-\s}x\right) \quad & \hbox{if} \; \; \s <0,
		\end{array} \right.
\end{equation}
where  $A,B \in \R$ are arbitrary constants. Moreover, multiplying by $u$ and integrating by parts \eqref{3.1} in $[0,L]$ yields to
\begin{equation}
\label{3.2}
-u(L)u'(L) + u(0)u'(0) + \int_0^L (u')^2 \, dx= \s \int_0^L u^2 \, dx.
\end{equation}
The proof is divided into 6 different cases.

\subsection{Case 1: $\mc{B}_0, \mc{B}_L \in \{ \mc{D}, \mc{N} \}$}
Suppose $\mc{B}_0=\mc{B}_L=\mc{N}$, i.e., $\b_0=\b_\o=0$. Then, $\s_1[-D^2,0,0,(0,L)]=0$, regardless the size of $L>0$,  and, hence, in this case, Part (b) holds true. Thus, throughout the remaining of this case we will assume that
$\{\mc{B}_0,\mc{B}_L\}\neq \{\mc{N}\}$. Hence, $\b_0,\b_\o \in \{0,+\infty \}$ with $(\b_0,\b_\o)\neq (0,0)$.
Since $\mc{B}_0, \mc{B}_L \in \{ \mc{D}, \mc{N} \}$,
$$
   -u(L)u'(L) + u(0)u'(0) = 0.
$$
Consequently, it follows from  \eqref{3.2} that any eigenvalue $\s$ of \eqref{1.3} must satisfy $\s \geq 0$.
Suppose $\s=0$. Then, by \eqref{3.2}, $u'=0$ in $(0,L)$ and, hence, there is some constant $C\in\R$ such that $u(x)=C$ for all $x\in [0,L]$. In particular, $u(0)=u(L)=C$. Consequently, as we are assuming that $\mc{B}_0=\mc{D}$, or $\mc{B}_L=\mc{D}$, we find that $C=0$. Therefore, any eigenvalue of \eqref{1.3} must be positive. Furthermore, based
on the expression of the admissible eigenfunctions collected in \eqref{3.2} we are easily driven to the sequence of eigenvalues of \eqref{1.3}
\begin{equation}
	\s_n= \left \{ \begin{array}{ll}
		\left(\tfrac{n \pi}{L}\right)^2 \quad & \hbox{if } \;\; \mc{B}_0 = \mc{B}_L = \mc{D}, \\[1ex]
		\left(\tfrac{(2n-1) \pi}{2L}\right)^2 \quad & \hbox{if} \;\; \mc{B}_0 = \mc{D} \; \; \hbox{and} \; \; \mc{B}_L = \mc{N}, \quad \hbox{or} \quad \mc{B}_0 = \mc{N} \; \; \hbox{and} \; \; \mc{B}_L = \mc{D}, \end{array} \right.
\end{equation}
for $n = 1,2,3, \dots$ Thus,
\begin{equation}
	\s_1[-D^2, \b_0,\b_\o,(0,L)]= \left \{ \begin{array}{ll}
		\left(\tfrac{\pi}{L}\right)^2 \quad & \hbox{if } \;\; \mc{B}_0 = \mc{B}_L = \mc{D}, \\[1ex]
		\left(\tfrac{\pi}{2L}\right)^2 \quad & \hbox{if} \;\; \mc{B}_0 = \mc{D} \; \; \hbox{and} \; \; \mc{B}_L = \mc{N}, \quad \hbox{or} \quad \mc{B}_0 = \mc{N} \; \; \hbox{and} \; \; \mc{B}_L = \mc{D}. \end{array} \right.
\end{equation}
Therefore,
$$
\lim_{L \da 0} \s_1[-D^2, \b_0,\b_\o,(0,L)] = +\infty.
$$

\subsection{Case 2: $\mc{B}_0 = \mc{R}_\b$, with $\b_0 \in\R\setminus\{0\}$, and $\mc{B}_L = \mc{D}$}
In this case, the problem \eqref{1.3} reads as
\begin{equation}
\label{3.3}
	\left \{ \begin{array}{ll}
		-u''=\s u \qquad \hbox{in} \;\; [0,L],\\[1ex]
		-u'(0) + \b_0 u(0)= 0, \quad u(L) = 0. \end{array} \right.
\end{equation}
First, suppose that \eqref{3.3} admits the eigenvalue $\s = 0$. Then, since $u(x)=Ax+B$ for some constants $A, B\in\R$,
imposing the boundary conditions we are led to
$$
\b_0 = -\tfrac{1}{L} \;\; \hbox{and} \;\; u(x) = B\left(-\tfrac{x}{L}+1\right)
$$
for all $x\in [0,L]$ and $B\in\R$. Since $u(x)>0$ for all $x\in [0,L)$ with $B>0$, by the uniqueness
of the principal eigenvalue, it is apparent that, for every $L>0$,
$$
\s_1\left[-D^2, -\tfrac{1}{L},+\infty,(0,L)\right] = 0,
$$
with associated principal eigenfunction $-\tfrac{x}{L}+1$ for all $x\in [0,L]$. Thus, by Proposition \ref{pr2.1},
we find that, for every $\b_0\in\R\setminus\{0\}$,
$$
\s_1(L):= \s_1[-D^2, \b_0,+\infty,(0,L)] > 0 \quad \hbox{if either}\;\; \b_0>0,\;\; \hbox{or}\;\; \b_0<0\;\;
\hbox{and} \;\; L<-\tfrac{1}{\b_0}.
$$
Hence, as we are simply interested in ascertaining the behavior of $\s_1(L)$ as $L\da 0$, without loss
of generality, we can assume that $\s_1(L)>0$. So, the principal eigenfunction of \eqref{3.3}, denoted by $\psi_1$,
must be given by
$$
  \psi_1(x)=A \cos\left(\sqrt{\s_1(L)}\,x\right) + B \sin\left(\sqrt{\s_1(L)}\,x\right),\quad x\in [0,L],
$$
for the appropriate values of $A$ and $B$. Imposing the boundary conditions of \eqref{3.3}, it becomes apparent that
$$
  B\neq 0, \quad A = \tfrac{\sqrt{\s_1(L)}}{\b_0}B, \quad \tfrac{\sqrt{\s_1(L)}}{\b_0}\cos\left(\sqrt{\s_1(L)}L\right)+
  \sin \left(\sqrt{\s_1(L)}L\right)=0.
$$
From the last identity, it is easily seen that $\cos\left(\sqrt{\s_1(L)}L\right)\neq 0$. Consequently,
$\s_1(L)$ is the first positive root of
\begin{equation}
\label{3.4}
\tan\left(\sqrt{\s_1(L)}\,L\right) = -\tfrac{\sqrt{\s_1(L)}}{\b_0}.
\end{equation}
Setting $s:= \sqrt{\s_1(L)}\,L >0$, \eqref{3.4}  becomes
$$
\tan s = -\frac{s}{L \b_0}.
$$
Thus, $\s_1(L)=\left(\frac{s}{L}\right)^2$ is determined by the first positive intersection between the graphs of $\tan s$ and the straight line $-\frac{s}{L \b_0}$, as illustrated in Figure \ref{fig1}, where we are representing the line $-\frac{s}{L \b_0}$ in the two admissible cases when $\b_0<0$ and $\b_0>0$. In the first case, the first crossing point between $\tan s$ and the line $-\frac{s}{L \b_0}$ is $s(L)<\frac{\pi}{2}$, while $s(L)>\frac{\pi}{2}$ if $-\frac{1}{L\b_0}<0$. In each of these situations, it becomes apparent from the construction (see
Figure \ref{fig1}) that $\lim_{L\da 0}s(L)=\frac{\pi}{2}$. Consequently,
$$
\s_1(L)=\left(\tfrac{s(L)}{L}\right)^2 \sim \left(\tfrac{\pi}{2L}\right)^2 \quad \hbox{as} \; \, L \da 0.
$$
Therefore,
\begin{equation}
\lim_{L \da 0} \s_1[-D^2, \b_0,+\infty,(0,L)] = +\infty.
\end{equation}

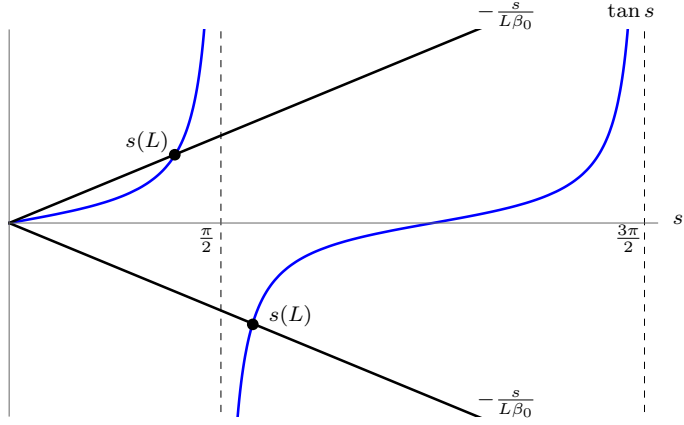
\begin{figure}[h!]
	\centering
		\begin{overpic}[scale=0.45]{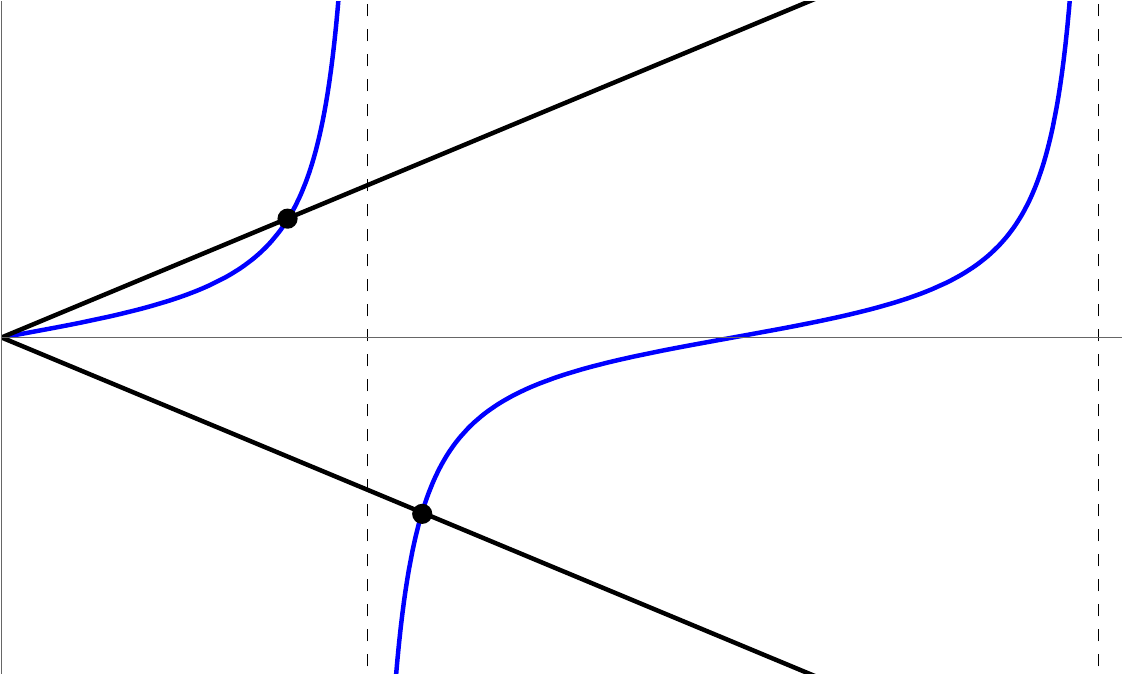}
		\put (18,42) {\scriptsize$s(L)$}
		\put (40,15) {\scriptsize$s(L)$}
		\put (29,27) {\scriptsize$\frac{\pi}{2}$}
		\put (93,27) {\scriptsize$\frac{3\pi}{2}$}
		\put (102,30) {\scriptsize$s$}
		\put (92,62) {\scriptsize$\tan s$}
		\put (72,62) {\scriptsize$-\frac{s}{L \b_0}$}
		\put (72,2) {\scriptsize$-\frac{s}{L \b_0}$}
	\end{overpic}
	\caption{The graphs of the line $-\frac{s}{L \b_0}$, in black, and $\tan s$, in blue.}
	\label{fig1}
\end{figure}

\subsection{Case 3: $\mc{B}_0 = \mc{D}$ and $\mc{B}_L = \mc{R}_\b$, with $\b_\o \in \R \setminus \{0\}$}
In this case, according to Proposition \ref{pr2.3}, we have that
$$
   \s_1[-D^2, +\infty,\b_\o,(0,L)]= \s_1[-D^2, \b_\o,+\infty,(0,L)]
$$
and, hence, it follows from Case  2 that
$$
\lim_{L \da 0} \s_1[-D^2, +\infty,\b_\o,(0,L)] = +\infty.
$$

\subsection{Case 4: $\mc{B}_0 = \mc{R}_\b$, with $\b_0 \in \R \setminus \{0\}$, and $\mc{B}_L = \mc{N}$}
In this case, the problem \eqref{1.3} becomes
\begin{equation}
\label{3.5}
	\left \{ \begin{array}{ll}
		-u''=\s u \qquad  \hbox{in} \;\; [0,L],\\[1ex]
		-u'(0) + \b_0 u(0)= 0, \quad u'(L) = 0. \end{array} \right.
\end{equation}
Since $\s_1[-D^2, 0,0,(0,L)]=0$, it follows from Proposition \ref{pr2.1} that
\begin{equation}
	\s_1(L) :=\s_1[-D^2, \b_0,0,(0,L)] \left \{ \begin{array}{ll}
		>0 \quad & \hbox{if} \;\; \b_0 >0,\\[1ex]
		<0 \quad & \hbox{if} \;\; \b_0 <0. \end{array} \right.
\end{equation}
Subsequently, we distinguish two different cases depending on the sign of $\b_0$.
\par
\medskip
\noindent \emph{Case $\b_0>0$.} Then, the principal eigenfunction $\psi_1$ of \eqref{3.5} inherits the form
\begin{equation}
\label{3.6}
  \psi_1(x)=A \cos\left(\sqrt{\s_1(L)}\,x\right) + B \sin\left(\sqrt{\s_1(L)}\,x\right),\quad x\in [0,L],
\end{equation}
for appropriate constants $A$ and $B$. Thus, adapting the argument of Case 2, we have that $\s_1(L)$ is the first positive solution of
$$
\tan \left(\sqrt{\s_1(L)}L\right) = \tfrac{\b_0}{\sqrt{\s_1(L)}},
$$
which, setting $s:= \sqrt{\s_1(L)}>0$, becomes
\begin{equation}
\label{3.7}
\tan \left(sL\right) = \tfrac{\b_0}{s}.
\end{equation}
Figure \ref{fig2} shows the plots of the curves $\tfrac{\b_0}{s}$, and $\tan (sL)$ for two different values of $L$.
\par
\vspace{3mm}
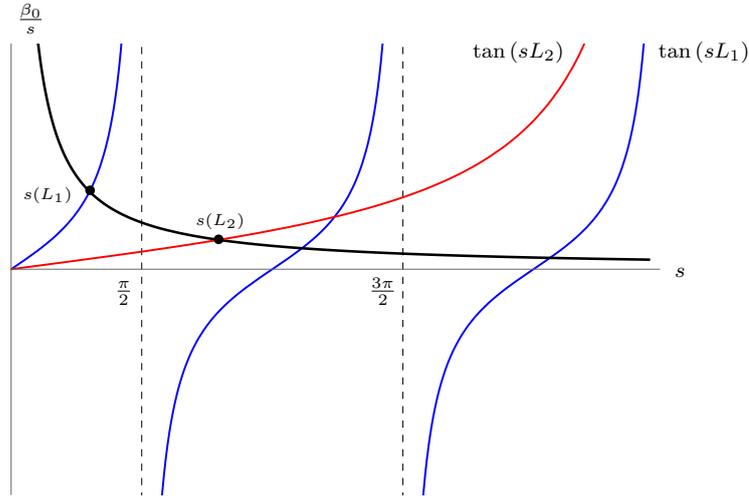
\begin{figure}[h!]
	\centering
	\begin{overpic}[scale=0.45]{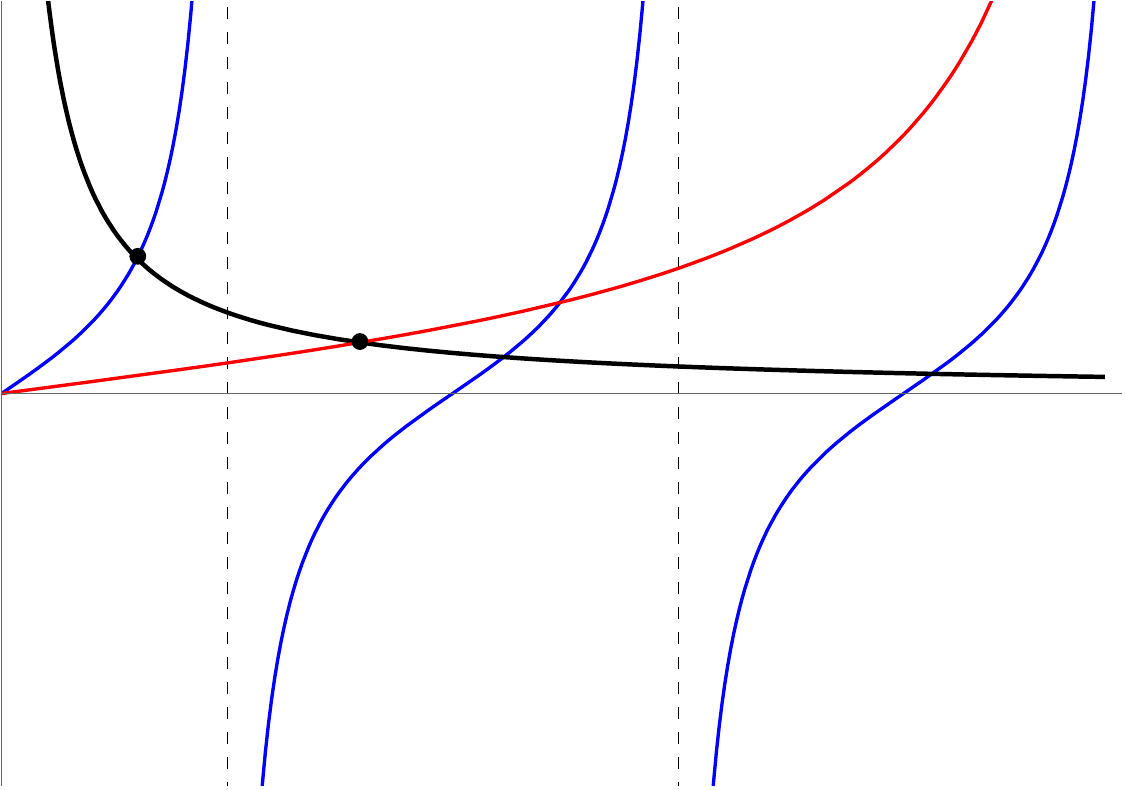}
		\put (2,46) {\tiny$s(L_1)$}
		\put (28.5,42) {\tiny$s(L_2)$}
		\put (16,31) {\scriptsize$\frac{\pi}{2}$}
		\put (55,31) {\scriptsize$\frac{3\pi}{2}$}
		\put (102,34) {\scriptsize$s$}
		\put (71,68) {\scriptsize$\tan \left(sL_2\right)$}
		\put (99.5,68) {\scriptsize$\tan \left(sL_1\right)$}
		\put (1,73) {\scriptsize$\frac{\b_0}{s}$}
	\end{overpic}
	\caption{The graphs of the hyperbola $\tfrac{\b_0}{s}$, in black,  and the curve $\tan (sL)$ for
$L_1=1$, in blue, and for $L_2=0.2$, in red. }
	\label{fig2}
\end{figure}
Since
\begin{equation}
\label{3.8}
  \lim_{L\da 0} \tan(sL)=0 \quad \hbox{uniformly in}\;\; s\in [0,M]\;\; \hbox{for all} \;\; M>0,
\end{equation}
it becomes apparent that the smaller $L>0$ the flatter is the corresponding function $\tan(sL)$, as illustrated by Figure \ref{fig2}, and, so, it needs a larger time to reach the hyperbola $\tfrac{\b_0}{s}$. Consequently,
$\lim_{L\da 0}s(L)=+\infty$, where $s(L)$ stands for the first positive root of \eqref{3.7}. Therefore,
\begin{equation}
\label{3.9}
\lim_{L \da 0} \s_1[-D^2, \b_0,0,(0,L)] = +\infty.
\end{equation}

\par
\medskip
\noindent \emph{Case $\b_0<0$.} Now, $\s_1(L)<0$ and the principal eigenfunction $\psi_1$ is given by
\begin{equation}
\label{3.10}
\psi_1(x) = A \cosh\left(\sqrt{-\s_1(L)}\,x\right) + B \sinh\left(\sqrt{-\s_1(L)}\,x\right),\quad x\in [0,L],
\end{equation}
for the appropriate values of $A$ and $B$. As in the previous cases, imposing the boundary operators $\mc{B}_0$ and $\mc{B}_L$ to $\psi_1$, we find that $\s_1(L)<0$ is the unique negative root of
\begin{equation}
\tanh \left(\sqrt{-\s_1(L)}\,L\right) = -\tfrac{\b_0}{\sqrt{-\s_1(L)}},
\end{equation}
which can be expressed, equivalently, as
$$
\tanh \left(sL\right) = -\tfrac{\b_0}{s},
$$
where $s:= \sqrt{-\s_1(L)}>0$. Figure \ref{fig3} plots the graphs of the hyperbola $-\tfrac{\b_0}{s}$ and the function $\tanh(sL)$ for two values of $L$.

\begin{figure}[h!]
	\centering
	\begin{overpic}[scale=0.45]{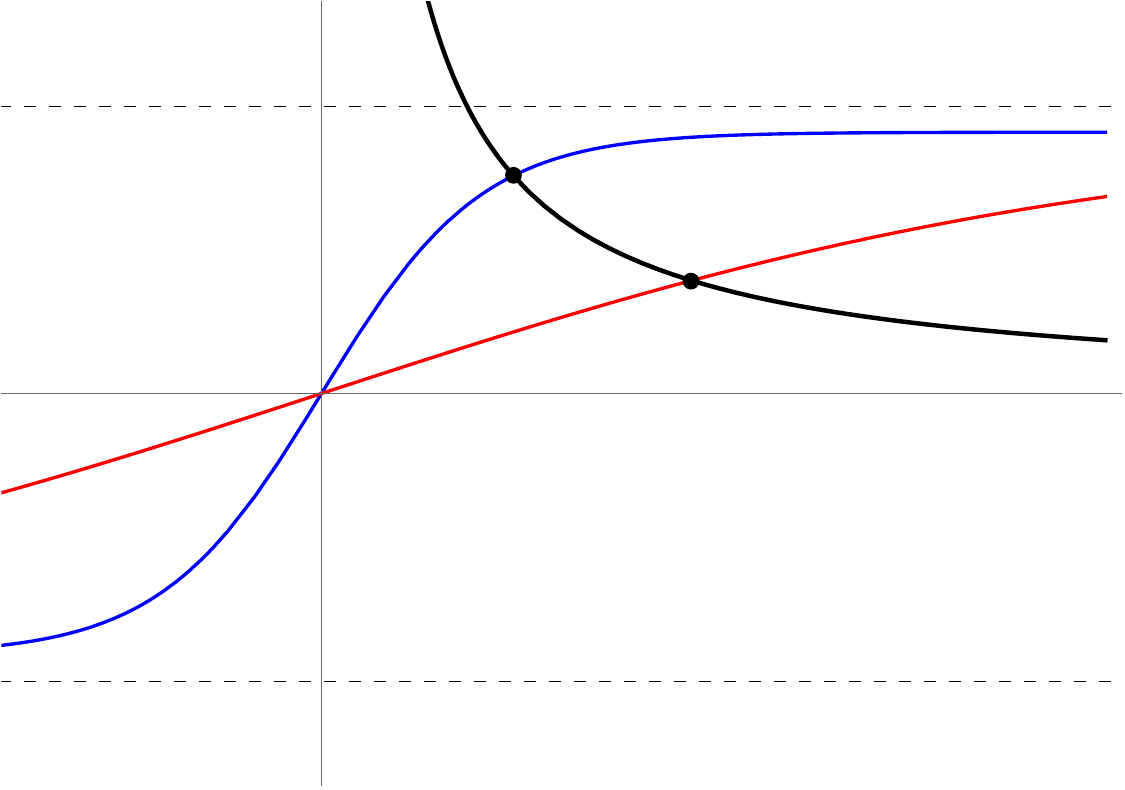}
		\put (35.7,54) {\tiny$s(L_1)$}
		\put (58,47.5) {\tiny$s(L_2)$}
		\put (102,34) {\scriptsize$s$}
		\put (100,58) {\scriptsize$\tanh \left(sL_1\right)$}
		\put (100,52) {\scriptsize$\tanh \left(sL_2\right)$}
		\put (32,73) {\scriptsize$-\frac{\b_0}{s}$}
	\end{overpic}
	\caption{The plots of the hyperbola $-\tfrac{\b_0}{s}$, in black, and $\tanh(sL)$ for $L_1=1$, in blue, and $L_2=0.2$, in red.}
	\label{fig3}
\end{figure}
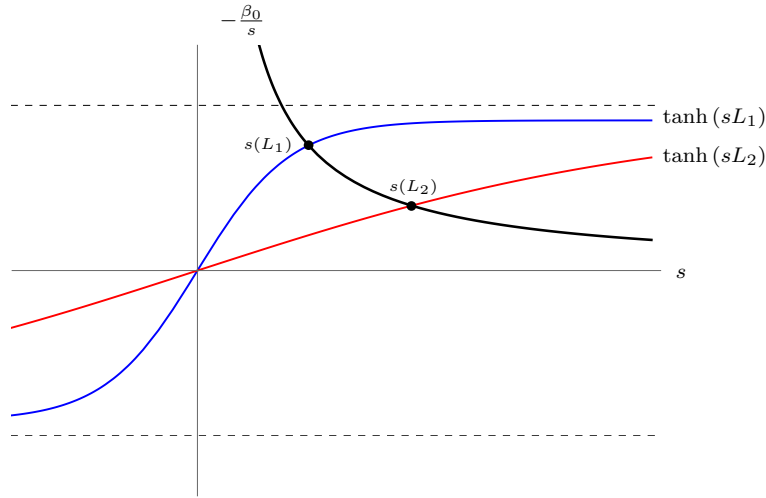

Since
\begin{equation}
\label{3.11}
  \lim_{L\da 0} \tanh(sL)=0 \quad \hbox{uniformly in}\;\; s\in [0,M]\;\; \hbox{for all} \;\; M>0,
\end{equation}
it becomes apparent that  $\lim_{L\da 0}s(L)=+\infty$ (see Figure \ref{fig3}). Therefore,
\begin{equation}
\label{3.12}
\lim_{L \da 0} \s_1[-D^2, \b_0,0,(0,L)] = -\infty.
\end{equation}

\subsection{Case 5: $\mc{B}_0 = \mc{N}$ and $\mc{B}_L = \mc{R}_\b$, with $\b_\o \in \R \setminus \{0\}$}

As in Case 3, by Proposition \ref{pr2.3}, we have that
$$
   \s_1[-D^2, 0,\b_\o,(0,L)]= \s_1[-D^2, \b_\o,0,(0,L)]
$$
and, hence, it follows from Case  4 that
$$
\lim_{L \da 0} \s_1[-D^2, 0,\b_\o,(0,L)] = \left\{\begin{array}{l} +\infty \quad \hbox{if} \; \; \b_\o >0,\\
-\infty \quad \hbox{if} \; \; \b_\o <0.\end{array}\right.
$$

\subsection{Case 6: $\mc{B}_0 = \mc{B}_L = \mc{R}_\b$, with $\b_0,\b_\o \in \R \setminus \{0\}$}
In this case, \eqref{1.3} reads as
\begin{equation}
\label{3.13}
	\left \{ \begin{array}{ll}
		-u''=\s u \qquad  \hbox{in} \;\; [0,L],\\[1ex]
		-u'(0) + \b_0 u(0)= 0, \quad u'(L) +\b_\o u(L) = 0, \end{array} \right.
\end{equation}
with $\b_0,\b_\o \in \R \setminus \{0\}$. First, suppose that \eqref{3.13} admits the eigenvalue $\s = 0$. Then, since $u(x) = Ax +B$ for some constants $A,B \in\R$, the boundary conditions entail that
\begin{equation}
\label{3.14}
\b_0 +\b_\o\left(\b_0 L +1\right) = 0,
\end{equation}
and, hence,
$$
\b_\o = \tfrac{-\b_0}{\b_0 L +1} \quad \hbox{and} \quad u(x) = B(\b_0 x+1)
$$
for all $x \in [0,L]$ and $B \in\R$. Note that $\b_0 L +1 \neq 0$, because, by \eqref{3.14}, $\b_0 L +1 =0$ implies
that $\b_0=0$, which contradicts the assumption $\b_0\neq 0$.
\par
Suppose $\b_0>0$. Then, for every $B>0$, $u(x)>0$ for all $x\in [0,L]$. Thus, by the uniqueness of the principal eigenvalue,
\begin{equation}
\label{3.15}
\s_1\left[-D^2,\b_0,\tfrac{-\b_0}{\b_0 L +1},(0,L)\right] = 0 \quad \hbox{for all} \; \; L >0.
\end{equation}
Suppose $\b_0<0$. Then, for every $L<-\frac{1}{\b_0}$ and $B >0$, we also have that $u(x) >0$ for all $x\in [0,L]$. Hence, once again by the uniqueness of the principal eigenvalue,
$$
\s_1\left[-D^2,\b_0,\tfrac{-\b_0}{\b_0 L +1},(0,L)\right] = 0 \quad \hbox{for all} \; \; L \in \left(0,-\tfrac{1}{\b_0}\right).
$$
Subsequently, we will differentiate several situations, depending on the signs of $\b_0$ and $\b_\o$.
\par
\vspace{0.2cm}
\noindent \emph{Case $\b_0 >0$.} Then, since $\tfrac{-\b_0}{\b_0 L +1}>-\b_0$, by Proposition \ref{pr2.1}, it follows from \eqref{3.15} that
\begin{equation}
\label{3.16}
\s_1(L):=\s_1[-D^2,\b_0,\b_\o,(0,L)]<0 \quad \hbox{if} \; \; \b_\o \leq -\b_0.
\end{equation}
Moreover, since
\begin{equation}
\label{3.17}
\tfrac{-\b_0}{\b_0 L +1} \da -\b_0 \quad \hbox{as} \; \; L \da 0,
\end{equation}
when $\b_\o>-\b_0$, there exists $L_0=L_0(\b_0,\b_\o)>0$ such that
$$
\b_\o > \tfrac{-\b_0}{\b_0 L+1}\quad \hbox{for all} \; \; L\in(0,L_0).
$$
Thus, again by Proposition \ref{pr2.1}, it follows from \eqref{3.15} that
\begin{equation}
\label{3.18}
\s_1(L):=\s_1[-D^2,\b_0,\b_\o,(0,L)]>0 \quad
\hbox{if} \; \; \b_\o > -\b_0 \quad \hbox{for all} \; \; L < L_0.
\end{equation}
Suppose that $\b_\o \geq 0$. Then, by Proposition \ref{pr2.1}, \eqref{3.9} yields to
$$
\lim_{L \da 0} \s_1[-D^2,\b_0,\b_\o,(0,L)]\geq \lim_{L \da 0} \s_1[-D^2,\b_0,0,(0,L)] = +\infty.
$$
Now, suppose that $-\b_0 < \b_\o <0$. Then, owing \eqref{3.18}, $\psi_1$ is given by \eqref{3.6} for appropriate
values of $A,B \in \R$. Hence, imposing the boundary conditions of \eqref{3.13}, taking into account that
$$
   \b_0\b_\o-\s_1(L)<0\;\;\hbox{for all}\;\; L\in (0,L_0),
$$
and adapting the argument of Case 2,  it follows that $\s_1(L)$ must be the first positive solution of
$$
\tan\left(\sqrt{\s_1(L)}\,L\right) = \frac{\sqrt{\s_1(L)}\left(\b_0+\b_\o\right)}{\s_1(L) -\b_0\b_\o}.
$$
Setting $s:=\sqrt{\s_1(L)}$, this equation becomes
$$
  \tan (sL)= \frac{s(\b_0+\b_\o)}{s^2-\b_0\b_\o}.
$$
Since the function
$$
   f(s):=\frac{s(\b_0+\b_\o)}{s^2-\b_0\b_\o}, \quad s\geq 0,
$$
satisfies $f(0)=0$, $f(s)>0$ for all $s>0$,
$\lim_{s\ua +\infty}f(s)=0$ and
$$
  f'(0)=-\frac{\b_0+\b_\o}{\b_0\b_\o}>0,
$$
it follows from \eqref{3.8} that
$$
\lim_{L \da 0} \s_1[-D^2,\b_0,\b_\o,(0,L)] = +\infty.
$$
Suppose that $\b_\o \leq -\b_0$. Then, thanks to \eqref{3.16}, $\psi_1$ is given by \eqref{3.10} for the appropriate values of $A ,B \in\R$. Assume, in addition, that $\b_\o = -\b_0$. Then, imposing the boundary conditions of \eqref{3.13} yields to
$$
\left(\s_1(L)+\b_0^2\right) \sinh\left(\sqrt{-\s_1(L)}\,L\right) = 0,
$$
and, hence, we have that
$$
\s_1(L)=\s_1[-D^2,\b_0,-\b_0,(0,L)]  = -\b_0^2 = \b_0\b_\o <0 \quad \hbox{for all} \; \; L>0.
$$
Subsequently, we assume that $\b_\o < -\b_0$. Then, the boundary conditions imply that $\s_1(L)$ is the unique positive solution of the equation
$$
\tanh\left(\sqrt{-\s_1(L)}L\right) = \frac{\sqrt{-\s_1(L)}\left(\b_0+\b_\o\right)}{\s_1(L) - \b_0\b_\o}.
$$
Finally, setting $s:=\sqrt{-\s_1(L)}$, this equation becomes
\begin{equation}
\label{3.19}
  \tanh (sL)= -\frac{s(\b_0+\b_\o)}{s^2+\b_0\b_\o}=\frac{s|\b_0+\b_\o|}{s^2+\b_0\b_\o},\quad s>0.
\end{equation}
In this case, since $\b_0\b_\o<0$, the function
$$
   g(s):=\frac{s|\b_0+\b_\o|}{s^2+\b_0\b_\o}, \quad s\geq 0,
$$
has a vertical asymptote at $s=s_0\equiv \sqrt{-\b_0\b_\o}>0$, and $g(s)<0$ for all $0<s<s_0$, whereas
$g(s)>0$ for all $s>s_0$. Thus, having a glance at the graphs of these functions
 it is easily realized that \eqref{3.19} has a unique positive root, $s>s_0$, and it follows from \eqref{3.11} that $s=s(L)\ua +\infty$ if $L\da 0$. Therefore,
$$
\lim_{L \da 0} \s_1[-D^2,\b_0,\b_\o,(0,L)] = -\infty.
$$
\par
\medskip
\noindent
\emph{Case $\b_0 <0$.} Thanks to \eqref{3.17} and reasoning as in the case $\b_0>0$, it becomes apparent that
\eqref{3.16} and \eqref{3.18} again hold true for all $L < -\tfrac{1}{\b_0}$. Suppose, in addition, that $\b_\o \leq 0$. Then, by Proposition \ref{pr2.1}, \eqref{3.12} implies that
$$
\lim_{L \da 0} \s_1[-D^2,\b_0,\b_\o,(0,L)] = -\infty \quad \hbox{for all} \; \; \b_\o \leq 0.
$$
So, we are done. Now, assume that $\b_\o>0$. Then, arguing as in case $\b_0 >0$, and again
based on \eqref{3.8} and \eqref{3.11}, it is easily realized that
$$
\lim_{L \da 0} \s_1[-D^2,\b_0,\b_\o,(0,L)] = +\infty \quad \hbox{if} \; \; \b_\o > -\b_0
$$
and
$$
\lim_{L \da 0} \s_1[-D^2,\b_0,\b_\o,(0,L)] = -\infty \quad \hbox{if} \; \; 0<\b_\o <-\b_0.
$$
By repetitive, the technical details of the proofs are omitted here. Moreover, Proposition \ref{pr2.3} entails that
$$
\s_1[-D^2,\b_0,-\b_0,(0,L)] = -\b_0^2 <0 \quad \hbox{for all} \; \; L>0 \quad \hbox{if} \; \; \b_\o = -\b_0.
$$
This ends the proof of Theorem \ref{th1.1}.

\section{Proof of  Theorem \ref{th1.2}}
\label{sec4}

\subsection{Proof of (\ref{1.4})}
The proof of \eqref{1.4} relies primarily on the Faber--Krahn inequality established by  Daners in \cite{D06}, and it uses a number of
comparisons and auxiliary results of a rather technical nature.  The main goal is analyzing the limiting behavior
of $\s_1[-\D,\b,\O]$ as $|\O|\da 0$ in the special case when
\begin{equation}
\label{4.1}
   \b_\ell:=\inf_{D}\b >0.
\end{equation}
This entails that, for every bounded domain $\O \subset D$ with $\bar \O \subset D$,
$$
  \min_{\p\O}\b \geq \b_\ell >0.
$$
Note that $ \min_{\p\O}\b$ is attaint because $\p \O\subset D$ is compact.
Let $\v_1\gg 0$ be a principal eigenfunction associated to $\s_1[-\D,\b,\O]$. Then,
\begin{equation}
	\left \{ \begin{array}{ll}
		-\D \v_1= \s_1[-\D,\b,\O] \v_1 \quad & \hbox{in} \;\; \O,\\[1ex]
		\frac{\p \v_1}{\p n} + \b \v_1 = 0 \quad & \hbox{on} \; \; \p \O. \end{array} \right.
\end{equation}
 Thanks to \eqref{4.1}, it follows from Proposition \ref{pr2.2} that
\begin{equation}
\label{4.2}
\s_1[-\D,\b,\O] \geq  \s_1[-\D,\b_\ell,\O]> \s_1[-\D,0,\O]=0.
\end{equation}
Moreover, by Theorem 1.1 of Daners \cite{D06}, the ball of radius $R>0$, $B_R$, has the smallest principal eigenvalue within all the bounded domains $\O$ such that $|\O| = |B_R|$. In other words,
\begin{equation}
\label{4.3}
\s_1[-\D,\b_\ell,\O] \geq \s_1[-\D,\b_\ell,B_R] \quad \hbox{if} \; \; |\O| = |B_R|,
\end{equation}
where $\s_1[-\D,\b_\ell,B_R] >0$ is the principal eigenvalue of the problem
\begin{equation}
\label{4.4}
	\left \{ \begin{array}{ll}
		-\D \t= \s \t \quad & \hbox{in} \;\; B_R,\\[1ex]
		\frac{\p \t}{\p n} + \b_\ell \t= 0 \quad & \hbox{on} \; \; \p B_R. \end{array} \right.
\end{equation}
Consequently, owing to \eqref{4.2} and \eqref{4.3}, in order to prove \eqref{1.4} it suffices to show that
\begin{equation}
\label{4.5}
\lim_{R \da 0} \s_1[-\D,\b_\ell,B_R] = +\infty
\end{equation}
because, since $|\O| = |B_R|$, $R \da 0$ entails $|\O| \da 0$.
\par
To prove \eqref{4.5} we can proceed as follows.  Let $\t\gg0$ be a principal eigenfunction associated to $\s_1[-\D,\b_\ell,B_R]$. As the support domain is a ball, $\t(x)$ must be radially symmetric. Thus, there exists a
function $\xi(r)$, with $\xi'(0) = 0$, such that
$$
  \t(x)= \xi(r), \quad \hbox{with} \; \; r := |x| = \sqrt{x_1^2 + \cdots + x_N^2}.
$$
So, for every $r \in (0,R]$ and $j\in\{1,...,N\}$, we have that
\begin{equation}
\tfrac{\p \t}{\p x_j}(x) = \xi'(r) \tfrac{\p r}{\p x_j} = \xi'(r) \tfrac{x_j}{r},\quad
\tfrac{\p^2 \t}{\p x_j^2}(x)=\xi''(r)\left(\tfrac{x_j}{r}\right)^2+\xi'(r)\frac{r-\frac{x_j^2}{r}}{r^2},
\end{equation}
and, hence, it is apparent that
$$
\nabla_x \t(x) = \xi'(r) \tfrac{x}{r}, \qquad \D_x \t(x) = \xi''(r) + \tfrac{N-1}{r}\xi'(r).
$$
Furthermore, $r=R$ if $x \in \p B_R$. Thus,
$$
\tfrac{\p \t}{\p n}(x) = \left\langle \nabla_x \t(x), \tfrac{x}{R}\right\rangle = \left\langle \xi'(R)\tfrac{x}{R}, \tfrac{x}{R}\right\rangle = \xi'(R).
$$
Consequently, \eqref{4.4} becomes into the next one-dimensional problem
\begin{equation}
\label{4.6}
	\left \{ \begin{array}{ll}
		 -\xi''(r) - \tfrac{N-1}{r}\xi'(r) =\s_1[-\D,\b_\ell,B_R] \xi(r) \qquad  \hbox{in} \;\; (0,R],\\[1ex]
		\xi'(0)= 0, \quad  \xi'(R) + \b_\ell \xi(R) = 0. \end{array} \right.
\end{equation}
\par
Next, we consider the associated one-dimensional eigenvalue problem
\begin{equation}
\label{4.7}
	\left \{ \begin{array}{ll}
		-\zeta'' =\s_1[-D^2,0,\b_\ell,(0,R)] \zeta \qquad  \hbox{in} \;\; [0,R],\\[1ex]
		\zeta'(0)= 0, \quad \zeta'(R) + \b_\ell \zeta(R) = 0,\end{array} \right.
\end{equation}
where the convection term in \eqref{4.6} is removed. Let $\zeta\gg 0$ be the principal eigenfunction of
\eqref{4.7} such that $\zeta(0)=1$. We claim that
\begin{equation}
\label{4.8}
\zeta'(r) < 0 \quad \hbox{for all} \; \; r \in (0,R].
\end{equation}
Indeed, since $\b_\ell>0$, it follows from Proposition \ref{pr2.1} that
$$
   \s_1(R):=\s_1[-D^2,0,\b_\ell,(0,R)] > \s_1[-D^2,0,0,(0,R)]= 0.
$$
Thus, $\zeta$ is given through
$$
\zeta(r)=A \cos\left(\sqrt{\s_1(R)}\,r\right) + B \sin\left(\sqrt{\s_1(R)}\,r\right)
$$
for some constants $A,B \in\R$. As $\zeta'(0)=0$, necessarily
$$
\zeta(r) = \cos\left(\sqrt{\s_1(R)}\,r\right),\quad r\in [0,R].
$$
Therefore, since $\zeta(0)=1$ and $\zeta(r) >0$ for all $r \in [0,R]$, \eqref{4.8} holds.
\par
Thanks to \eqref{4.8}, we find that
$$
 - \zeta''(r)-\tfrac{N-1}{r}\zeta'(r) - \s_1[-D^2,0,\b_\ell,(0,R)] \zeta(r) > -\zeta''(r) - \s_1[-D^2,0,\b_\ell,(0,R)]\zeta(r) =0
$$
for all $r \in (0,R]$. Consequently, $\zeta(r)$ is a positive strict supersolution of the problem
\begin{equation}
	\left \{ \begin{array}{ll}
		-u''-\tfrac{N-1}{r}u' =\s_1[-D^2,0,\b_\ell,(0,R)] u \qquad \hbox{in} \;\; (0,R],\\[1ex]
		u'(0)= 0, \quad u'(R) + \b_\ell u(R) = 0. \end{array} \right.
\end{equation}
Thus, it follows from Theorem 7.10 of \cite{LG13} that
$$
  \s_1\left[-D^2-\tfrac{N-1}{r}D-\s_1[-D^2,0,\b_\ell,(0,R)],0,\b_\ell,(0,R)\right]>0.
$$
Hence,  due to \eqref{4.6}, we find that
\begin{equation}
\label{4.9}
\s_1[-\D,\b_\ell,B_R] > \s_1[-D^2,0,\b_\ell,(0,R)].
\end{equation}
Finally, since $\b_\ell >0$, Theorem \ref{th1.1} guarantees that
$$
\lim_{R \da 0}\s_1[-D^2,0,\b_\ell,(0,R)] = +\infty.
$$
Therefore, \eqref{4.5} holds from \eqref{4.9}. This ends the proof of \eqref{1.4}.
\par
Note that, owing to Proposition 8.1 of \cite{LG13}, we already know that
$$
\s_1[-\D,\b,\O] < \s_1[-\D,\mc{D},\O] \quad \hbox{for all} \; \; \b \in \mc{C}(D).
$$
Consequently, \eqref{1.4} implies that
\begin{equation}
  \lim_{|\O|\da 0}\s_1[-\D,\mc{D},\O]=+\infty.
\end{equation}

\subsection{Proof of (\ref{1.5})}
Throughout this section,
$$
   \b_m:=\sup_{D}\b <0.
$$
This implies that, for every bounded domain $\O$ of $D$, with $\bar \O \subset D$, one has that
\begin{equation}
\label{4.10}
   \max_{\p\O}\b \leq \sup_D \b =\b_m <0.
\end{equation}
Let $\v_1\gg 0$ be a principal eigenfunction associated to $\s_1[-\D,\b,\O]$. Then,
\begin{equation}
	\left \{ \begin{array}{ll}
		-\D \v_1= \s_1[-\D,\b,\O] \v_1 \quad & \hbox{in} \;\; \O,\\[1ex]
		\tfrac{\p \v_1}{\p n} + \b \v_1 = 0 \quad & \hbox{on} \; \; \p \O. \end{array} \right.
\end{equation}
Thanks to Proposition \ref{pr2.2},  it follows from \eqref{4.10}  that
$$
\s_1[-\D,\b,\O] \leq  \s_1[-\D,\b_m,\O]<\s_1[-\D,0,\O] = 0.
$$
Thus, to prove \eqref{1.5}, it suffices to show that
\begin{equation}
\label{4.11}
\lim_{|\O| \da 0}  \s_1[-\D,\b_m,\O] = -\infty.
\end{equation}
According to Theorems 3.1 and 4.1 of Nardi \cite{N14}, the boundary value problem
\begin{equation}
\label{4.12}
	\left \{ \begin{array}{ll}
		-\D \phi = -1 \quad & \hbox{in} \;\; \O,\\[1ex]
		\tfrac{\p \phi}{\p n} = \phi_0 >0 \quad   & \hbox{on} \;\; \p \O,  \end{array} \right.
\end{equation}
where $\phi_0$ is a constant, possesses a solution $\phi \in \mc{C}^{2,\a}(\bar{\O})$, unique up to an additive constant, if and only if
\begin{equation}
\label{4.13}
\phi_0 = \frac{|\O|}{{\rm Area}(\p \O)}.
\end{equation}
Suppose \eqref{4.13}, let $\phi$ be a solution of \eqref{4.12}, and consider the function
\begin{equation}
\label{4.14}
\vt_1:= \frac{\Phi_1}{h} \quad \hbox{with} \; \; h := e^{M \phi},
\end{equation}
where $\Phi_1\gg 0$ is any principal eigenfunction associated to $\s_1[-\D,\b_m,\O]$ and $M>0$ is an arbitrary constant
to be chosen later. Then,
\begin{equation}
\label{4.15}
	\left \{ \begin{array}{ll}
		-\D \Phi_1= \s_1[-\D,\b_m,\O] \Phi_1 \quad & \hbox{in} \;\; \O,\\[1ex]
		\frac{\p \Phi_1}{\p n} + \b_m \Phi_1 = 0 \quad & \hbox{on} \; \; \p \O, \end{array} \right.
\end{equation}
and differentiating \eqref{4.14} yields to $\nabla h= Mh \nabla \phi$. Thus,
\begin{equation}
\label{4.16}
\D h = {\rm div}\left(Mh \nabla \phi\right) = M \langle \nabla h, \nabla \phi \rangle + Mh\D\phi = M^2 h |\nabla \phi|^2 +Mh \D \phi.
\end{equation}
On the other hand,
\begin{equation}
\label{4.17}
\begin{split}
  \s_1[-\D,\b_m,\O] \vt_1 h & = \s_1[-\D,\b_m,\O] \Phi_1 = - \D \Phi_1 =-\D(h\vt_1) \\ &
  =-h \D \vt_1 -2 \langle \nabla h, \nabla \vt_1 \rangle - \vt_1 \D h
\end{split}
\end{equation}
and, hence, since $h>0$ in $\O$, it follows from \eqref{4.16} and \eqref{4.17} that
\begin{equation}
\label{4.18}
-\D \vt_1 -2M \langle \nabla \phi, \nabla \vt_1 \rangle - M \left(M |\nabla \phi|^2 + \D \phi \right)\vt_1  = \s_1[-\D,\b_m,\O]  \vt_1.
\end{equation}
Moreover, on $\p\O$, we have that
\begin{align*}
\tfrac{\p \Phi_1}{\p n} &  = \langle\nabla\Phi_1,n\rangle = \langle\nabla(\vt_1h),n\rangle \\ & = \langle \vt_1Mh\nabla\phi+h\nabla\vt_1,n\rangle = \vt_1 Mh \tfrac{\p \phi}{\p n}+ h \tfrac{\p \vt_1}{\p n}.
\end{align*}
Consequently, since $h>0$ on $\p\O$, it follows from the boundary condition
of \eqref{4.15} that
\begin{equation}
\label{4.19}
\tfrac{\p \vt_1}{\p n} + \left(M \; \tfrac{\p \phi}{\p n} + \b_m\right) \vt_1= 0.
\end{equation}
Therefore, thanks to \eqref{4.18} and \eqref{4.19}, we find that
\begin{equation}
\label{4.20}
	\left \{ \begin{array}{ll}
		-\D\vt_1 -2M \langle \nabla \phi, \nabla \vt_1 \rangle - M \left(M |\nabla \phi^2| + \D\phi \right)\vt_1= \s_1[-\D,\b_m,\O] \vt_1 \quad & \hbox{in} \;\; \O,\\[2ex]
		\tfrac{\p \vt_1}{\p n} + \left(M \; \tfrac{\p \phi}{\p n}+ \b_m \right) \vt_1= 0 \quad & \hbox{on} \; \; \p \O. \end{array} \right.
\end{equation}
Since $\phi$ solves \eqref{4.12}, choosing
$$
M := -\b_m \left(\tfrac{\p \phi}{\p n}\right)^{-1} = -\tfrac{\b_m}{\phi_0} >0,
$$
it is apparent that the problem \eqref{4.20} becomes
\begin{equation}
\label{4.21}
\left \{ \begin{array}{ll}
		-\D\vt_1 + 2 \tfrac{\b_m}{\phi_0} \langle \nabla \phi, \nabla \vt_1 \rangle + \tfrac{\b_m}{\phi_0} \left(-\tfrac{\b_m}{\phi_0} |\nabla \phi^2| + 1\right)\vt_1= \s_1[-\D,\b_m,\O] \vt_1 \quad & \hbox{in} \;\; \O,\\[1ex]
		\tfrac{\p \vt_1}{\p n} = 0 \quad & \hbox{on} \; \; \p \O. \end{array} \right.
\end{equation}
Let denote by $\nu_1[\b_m,\O]$ the principal eigenvalue of
\begin{equation}
	\left \{ \begin{array}{ll}
		-\D u + 2 \tfrac{\b_m}{\phi_0} \langle \nabla \phi, \nabla u \rangle + \tfrac{\b_m}{\phi_0} u = \nu u \quad & \hbox{in} \;\; \O,\\[1ex]
		\tfrac{\p u} {\p n} = 0 \quad & \hbox{on} \; \; \p \O.\end{array} \right.
\end{equation}
As  $\tfrac{\b_m}{\phi_0}$ is a constant,
$$
  \nu_1[\b_m,\O]= \tfrac{\b_m}{\phi_0}
$$
and the constant function $1$ provides us with a principal eigenfunction of $\nu_1[\b_m,\O]$. Moreover, since
$$
\tfrac{\b_m}{\phi_0} > \tfrac{\b_m}{\phi_0} \left(-\tfrac{\b_m}{\phi_0} |\nabla \phi^2| + 1\right),
$$
we can infer from \eqref{4.21} that
\begin{equation}
	\left \{ \begin{array}{ll}
		\left( -\D  + 2 \tfrac{\b_m}{\phi_0} \langle \nabla \phi, \nabla  \rangle + \tfrac{\b_m}{\phi_0} -\s_1[-\D,\b_m,\O]\right) \vt_1>0  \quad & \hbox{in} \;\; \O,\\[1ex]
		\tfrac{\p \vt_1}{\p n} = 0 \quad & \hbox{on} \; \; \p \O. \end{array} \right.
\end{equation}
Consequently, thanks to Theorem 7.10 of \cite{LG13}, taking into account \eqref{4.13}, we find that
\begin{equation}
\nu_1[\b_m,\O]= \tfrac{\b_m}{\phi_0}=\b_m \tfrac{{\rm Area}(\p \O)}{|\O|} > \s_1[-\D,\b_m,\O].
\end{equation}
As $\b_m<0$, to prove \eqref{4.11} it suffices to show that
\begin{equation}
\label{4.22}
\lim_{|\O| \da 0} \; \frac{{\rm Area}(\p \O)}{|\O|} = +\infty.
\end{equation}
A proof of \eqref{4.22} can be derived as follows. According to the so called \emph{isoperimetric inequality}
(see, e.g., Ritor\'e and Sinestrari \cite{RS}), since $\O$ is of class $\mc{C}^1$,
$$
{\rm Area}(\p \O) \geq N \o_N^\frac{1}{N} |\O|^{\frac{N-1}{N}},
$$
where $\o_N:= \tfrac{\pi^{\frac{N}{2}}}{\Gamma\left(\frac{N}{2}+1\right)}$ stands for the Lebesgue measure of the unit ball in $\R^N$ ($\G$ denotes the Euler Gamma function). Thus,
$$
\frac{{\rm Area}(\p \O)}{|\O|} \geq N \o_N^{\frac{1}{N}} |\O|^{-\frac{1}{N}}.
$$
Therefore, letting $|\O|\da 0$, \eqref{4.22} holds true. This ends the proof of \eqref{1.5}.

\section{Asymptotic expansion of $\s_1[-\D,\b_\ell,B_R]$ as $R\da 0$ in case $\b_\ell>0$}
Let $\Phi_1\gg 0$ be a principal eigenfunction associated to $\s_1[-\D,\b_\ell,B_R]$. Then,
\begin{equation}
\label{5.1}
	\left \{ \begin{array}{ll}
		-\D \Phi_1=\s_1[-\D,\b_\ell,B_R] \Phi_1 \quad & \hbox{in} \;\; B_R,\\[1ex]
		\tfrac{\p \Phi_1}{\p n} + \b_\ell \Phi_1 = 0 \quad & \hbox{on} \; \; \p B_R. \end{array} \right.
\end{equation}
Moreover, the change of variable
\begin{equation}
\Psi_1(y) := \Phi_1(x), \qquad y := \tfrac{x}{R},
\end{equation}
transforms \eqref{5.1} into
\begin{equation}
\label{5.2}
	\left \{ \begin{array}{ll}
		-\D \Psi_1=R^2 \s_1[-\D,\b_\ell,B_R] \Psi_1 \quad & \hbox{in} \;\; B_1,\\[1ex]
		\tfrac{\p \Psi_1}{\p n} + \b_\ell R \Psi_1 = 0 \quad & \hbox{on} \; \; \p B_1. \end{array} \right.
\end{equation}
Since $\Psi_1\gg 0$, by the uniqueness of the principal eigenvalue, it becomes apparent that
\begin{equation}
\label{5.3}
\Sigma(R) := \s_1[-\D,\b_\ell R,B_1] = R^2 \s_1[-\D,\b_\ell,B_R] >0.
\end{equation}
The main result of this section can be stated as follows.

\begin{theorem}
\label{th5.1}
The mapping $R \mapsto \Sigma(R)$ is real analytic. Moreover,
\begin{equation}
\Sigma(R) = \b_\ell \frac{{\rm Area}(\p B_1)}{|B_1|}R + O(R^2)\quad \hbox{as}\;\;
R\da 0.
\end{equation}
Thus,
$$
  \lim_{R\da 0} \frac{\Sigma(R)}{R} =\b_\ell \frac{{\rm Area}(\p B_1)}{|B_1|},
$$
and, thanks to \eqref{5.3},
$$
  \s_1[-\D,\b_\ell,B_R]= \frac{\Sigma(R)}{R^2}=\b_\ell \frac{{\rm Area}(\p B_1)}{|B_1|}\frac{1}{R}+O(1)
  \quad \hbox{as}\;\; R\da 0.
$$
\end{theorem}
\begin{proof}
To show the analyticity of $\Sigma(R)$ we can argue as follows. The function
$$
  \phi(x):=\tfrac{1}{2}e^{1-|x|^2},\qquad x\in\R^N,
$$
is analytic and it satisfies
$$
   \tfrac{\p \phi}{\p n} = -1\quad \hbox{in}\;\; \p B_1,
$$
where $n$ stands for the outward unit
normal at $B_1$. Subsequently, for some constant $M>0$ to be chosen later, we consider the function
\begin{equation}
\label{5.4}
\vt_1:= \frac{\Psi_1}{h} \quad \hbox{with} \; \; h := e^{M \phi}.
\end{equation}
Then, arguing as in Section 4.2, it is apparent that \eqref{5.2} becomes
\begin{equation}
\label{5.5}
	\left \{ \begin{array}{ll}
		-\D \vt_1 -2M \langle \nabla \phi, \nabla \vt_1 \rangle -  \left(M^2 |\nabla \phi|^2 + M \D \phi \right)\vt_1 = \Sigma(R) \vt_1 \quad & \hbox{in} \;\; B_1,\\[1ex]
		\tfrac{\p \vt_1}{\p n} +  \left(-M + \b_\ell R \right)\vt_1 = 0 \quad & \hbox{on} \; \; \p B_1. \end{array} \right.
\end{equation}
Thus, choosing
$$
     M := \b_\ell R >0,
$$
the problem \eqref{5.5} reduces to
\begin{equation}
\label{5.6}
	\left \{ \begin{array}{ll}
		-\D \vt_1 -2\b_\ell R \langle \nabla \phi, \nabla \vt_1 \rangle - \b_\ell \left(\b_\ell R^2 |\nabla \phi|^2 +  R\D \phi \right)\vt_1 = \Sigma(R) \vt_1 \quad & \hbox{in} \;\; B_1,\\[1ex]
		\tfrac{\p \vt_1}{\p n} = 0 \quad & \hbox{on} \; \; \p B_1. \end{array} \right.
\end{equation}
Consequently, since the dependence of the operator in $R$ is analytic and $\Sigma(R)$ is an algebraically simple eigenvalue,  by Lemma 2.1.1 of \cite{LG01}, $\Sigma(R)$ must be analytic in $R\in\R$ (see Kato \cite{K95}, if necessary). Note that the parameter $R$ in \eqref{5.6} does not need to be positive, but can take any real value, much like in the problem
\begin{equation}
\label{5.7}
	\left \{ \begin{array}{ll}
		-\D \Psi_1=\Sigma(R) \Psi_1 \quad & \hbox{in} \;\; B_1,\\[1ex]
		\tfrac{\p \Psi_1}{\p n} + \b_\ell R \Psi_1 = 0 \quad & \hbox{on} \; \; \p B_1. \end{array} \right.
\end{equation}
Only when $R>0$, we are going to have \eqref{5.3}. So, $\Sigma(R)$ stands for the principal eigenvalue of the problem \eqref{5.6} (or \eqref{5.7}). By analyticity, we have that
$$
\Sigma(R) = \Sigma(0) + \dot \Sigma(0)R + O(R^2)\quad \hbox{as}\;\; R\to 0,
$$
where \; \lq\lq$\cdot$\rq\rq\; stands for  derivation with respect to $R$.
Moreover, since $\Sigma(0)=0$,
\begin{equation}
\Sigma(R) = \dot \Sigma(0)R + O(R^2)\quad \hbox{as}\;\; R\to 0.
\end{equation}
To complete the proof of the theorem it remains to show that
\begin{equation}
\label{5.8}
\dot \Sigma(0) = \b_\ell \frac{{\rm Area}(\p B_1)}{|B_1|}.
\end{equation}
According to Lemma 2.1.1 of \cite{LG01}, the eigenfunction $\vt_1\equiv \vt_1(R)$ can be also chosen to be
analytic in $R$. Moreover, differentiating \eqref{5.6} with respect to $R$, $\dot \vt_1(R)$  satisfies $\frac{\p \dot \vt_1}{\p n}(R)=0$ on $\p B_1$, and
\begin{equation}
\label{5.9}
\begin{split}
		-\D \dot \vt_1 -2\b_\ell \langle \nabla \phi, \nabla \vt_1 \rangle & -
2\b_\ell R \langle \nabla \phi, \nabla \dot \vt_1 \rangle
-  \b_\ell \left(\b_\ell 2 R |\nabla \phi|^2 +   \D \phi \right)\vt_1 \\ &
- \b_\ell \left(\b_\ell R^2 |\nabla \phi|^2 +  R \D \phi \right) \dot \vt_1
= \dot \Sigma(R) \vt_1 +\Sigma(R)\dot \vt_1
\end{split}
\end{equation}
in $B_1$. On the other hand, differentiating with respect to $R$ in \eqref{5.4}, we find that
\begin{equation}
\label{5.10}
  \dot \vt_1= \frac{\dot \Psi_1 h-\dot h\Psi_1}{h^2},\qquad \dot h=\b_\ell \phi e^{\b_\ell R \phi}=\b_\ell \phi h.
\end{equation}
Thus, substituting \eqref{5.10} into \eqref{5.9} and the boundary condition $\frac{\p \dot \vt_1}{\p n}=0$, after some straightforward, but rather tedious manipulations,
it becomes apparent that
\begin{equation}
\label{5.11}
	\left \{ \begin{array}{ll}
		-\D \dot \Psi_1= \dot \Sigma(R) \Psi_1 + \Sigma(R) \dot \Psi_1 \quad & \hbox{in} \;\; B_1,\\[1ex]
		\tfrac{\p \dot \Psi_1}{\p n}+ \b_\ell \Psi_1 + \b_\ell R \dot \Psi_1= 0 \quad & \hbox{on} \; \; \p B_1, \end{array} \right.
\end{equation}
which provides us with the formal variational problem from \eqref{5.7} with respect to the parameter $R\in\R$.
\par
Multiplying the differential equation of \eqref{5.11} by $\Psi_1$ and integrating in $B_1$ shows that
\begin{equation}
\label{5.12}
- \int_{B_1} \Psi_1 \D \dot \Psi_1 \, dy -  \Sigma(R) \int_{B_1} \Psi_1 \dot \Psi_1 \, dy = \dot \Sigma(R) \int_{B_1}  \Psi_1^2 \, dy.
\end{equation}
On the other hand, thanks to the Green's second identity,
$$
\int_{B_1} \dot \Psi_1 \D \Psi_1 \, dy -\int_{B_1} \Psi_1 \D \dot \Psi_1 \, dy = \int_{\p B_1} \dot \Psi_1 \tfrac{\p \Psi_1}{\p n} \, dS - \int_{\p B_1} \Psi_1 \tfrac{\p \dot \Psi_1}{\p n} \, dS.
$$
Thus, by the boundary conditions of \eqref{5.2} and \eqref{5.11},
\begin{equation}
\label{5.13}
\int_{B_1} \dot \Psi_1 \D \Psi_1 \, dy  -\int_{B_1} \Psi_1 \D \dot \Psi_1 \, dy  =  \int_{\p B_1} \b_\ell \Psi_1^2 \, dS.
\end{equation}
Consequently, taking into account the differential equation of \eqref{5.11}, it follows from \eqref{5.12} and \eqref{5.13} that
$$
\int_{\p B_1}\b_\ell \Psi_1^2 \, dS = \int_{B_1} \dot \Psi_1 \left(-\D \Psi_1 - \Sigma(R) \Psi_1\right) \, dy + \int_{\p B_1}\b_\ell \Psi_1^2 \, dS =  \dot \Sigma(R) \int_{B_1} \Psi_1^2 \, dy.
$$
Therefore, we can conclude that
\begin{equation}
\label{5.14}
\dot \Sigma(R) = \b_\ell \frac{ \int_{\p B_1} \Psi_1^2 \; dy}{\int_{B_1} \Psi_1^2 \; dy}.
\end{equation}
Finally, normalizing $\Psi_1$ so that
\begin{equation}
\label{5.15}
\int_{B_1} \Psi_1^2 \, dy = 1
\end{equation}
and using that $\Sigma(0) = 0$ and that $\Psi_1$ is constant for $R=0$, we can infer from \eqref{5.15} that actually
\begin{equation}
\label{5.16}
\Psi_1^2= \frac{1}{|B_1|}\quad \hbox{if}\;\; R=0.
\end{equation}
Consequently, it follows from \eqref{5.14}-\eqref{5.16} that
$$
\dot \Sigma(0) = \b_\ell \frac{{\rm Area}(\p B_1)}{|B_1|}.
$$
Thus, \eqref{5.8} holds. This ends the proof.
\end{proof}
\par
Observe that Theorem \ref{th5.1} entails that
$$
\lim_{R \da 0}\s_1[-\D,\b_\ell,B_R]= +\infty,
$$
which provides us with another proof of \eqref{1.4}.
\par

\end{document}